\documentclass[12pt]{amsart}
\usepackage{amscd,amsmath,amsthm,amssymb,graphics}
\usepackage{pstcol,pst-plot,pst-3d}
\usepackage{lmodern,pst-node}
\usepackage[dvips]{graphicx}
\usepackage{multicol}
\usepackage{epic,eepic}
\usepackage{amsfonts,amssymb,amscd,amsmath,enumitem,verbatim}
\psset{unit=0.7cm,linewidth=0.8pt,arrowsize=2.5pt 4}

\newpsstyle{fatline}{linewidth=1.5pt}
\newpsstyle{fyp}{fillstyle=solid,fillcolor=verylight}
\definecolor{verylight}{gray}{0.97}
\definecolor{light}{gray}{0.9}
\definecolor{medium}{gray}{0.85}
\definecolor{dark}{gray}{0.6}

\usepackage{lineno}

\def\frk{\frak}               

\def\Phi{{\frk n}}
\def\Phi{{\frk N}}
\def\MR{{\mathcal R}}

\def\MG{{\mathcal G}}

\def\opn#1#2{\def#1{\operatorname{#2}}} 
\opn\chara{char} \opn\length{\ell} \opn\pd{pd} \opn\rk{rk}
\opn\projdim{proj\,dim} \opn\injdim{inj\,dim} \opn\rank{rank}
\opn\depth{depth} \opn\grade{grade} \opn\height{height}
\opn\embdim{emb\,dim} \opn\codim{codim}

\opn\Tr{Tr} \opn\bigrank{big\,rank}
\opn\superheight{superheight}\opn\lcm{lcm}
\opn\trdeg{tr\,deg}
\opn\reg{reg} \opn\lreg{lreg} \opn\ini{in} \opn\lpd{lpd}
\opn\size{size}\opn\bigsize{bigsize}
\opn\cosize{cosize}\opn\bigcosize{bigcosize}
\opn\sdepth{sdepth}\opn\sreg{sreg}
\opn\link{link}\opn\fdepth{fdepth}
\opn\div{div} \opn\Div{Div} \opn\cl{cl} \opn\Cl{Cl}
\let\epsilon\varepsilon
\let\phi=\varphi
\let\kappa=\varkappa
\opn\Spec{Spec} \opn\Supp{Supp} \opn\supp{supp} \opn\Sing{Sing}
\opn\Ass{Ass} \opn\Min{Min}\opn\Mon{Mon} \opn\dstab{dstab} \opn\astab{astab}
\opn\Syz{Syz}
\opn\Ann{Ann} \opn\Rad{Rad} \opn\Soc{Soc}
\opn\Im{Im} \opn\Ker{Ker} \opn\Coker{Coker} \opn\Am{Am}
\opn\Hom{Hom} \opn\Tor{Tor} \opn\Ext{Ext} \opn\End{End}
\opn\Aut{Aut} \opn\id{id}

\opn\nat{nat}
\opn\pff{pf}
\opn\Pf{Pf} \opn\GL{GL} \opn\SL{SL} \opn\mod{mod} \opn\ord{ord}
\opn\Gin{Gin} \opn\Hilb{Hilb}\opn\sort{sort}
\opn\initial{init}
\opn\ende{end}
\opn\height{height}
\opn\type{type}
\opn\set{set}
\opn\aff{aff} \opn\con{conv} \opn\relint{relint} \opn\st{st}
\opn\lk{lk} \opn\cn{cn} \opn\core{core} \opn\vol{vol}
\opn\link{link} \opn\star{star}\opn\lex{lex}
\opn\gr{gr}

\def\pot#1#2{#1[\kern-0.28ex[#2]\kern-0.28ex]}

\opn\dirlim{\underrightarrow{\lim}}
\opn\inivlim{\underleftarrow{\lim}}
\let\to=\rightarrow

\def\Implies{\ifmmode\Longrightarrow \else
        \unskip${}\Longrightarrow{}$\ignorespaces\fi}
\def\implies{\ifmmode\Rightarrow \else
        \unskip${}\Rightarrow{}$\ignorespaces\fi}
\def\iff{\ifmmode\Longleftrightarrow \else
        \unskip${}\Longleftrightarrow{}$\ignorespaces\fi}

\let\:=\colon
 \theoremstyle{plain}
\newtheorem{Theorem}{Theorem}[section]
 
 \newtheorem{Corollary}[Theorem]{Corollary}
 \newtheorem{Proposition}[Theorem]{Proposition}

 \theoremstyle{definition}
 \newtheorem{Definition}[Theorem]{Definition}
 \newtheorem{Remark}[Theorem]{Remark}
 
 \newtheorem{Example}[Theorem]{Example}

\let\epsilon\varepsilon
\let\kappa=\varkappa
\opn\dis{dis}
\def\pnt{{\raise0.5mm\hbox{\large\bf.}}}

\opn\Lex{Lex}

\begin{document}
\title{Powers of $t$-spread principal Borel ideals}
\author {Claudia  Andrei, Viviana Ene, Bahareh Lajmiri}

\address{Claudia Andrei, Faculty of Mathematics and Computer Science, University of Bucharest, Str. Academiei 14, 
 010014 Bucharest, Romania} \email{claudiaandrei1992@gmail.com}

\address{Viviana Ene, Faculty of Mathematics and Computer Science, Ovidius University, Bd.\ Mamaia 124,
 900527 Constanta, Romania}  \email{vivian@univ-ovidius.ro}

\address{Bahareh Lajmiri, Department of Mathematics and Computer Science, Amirkabir University of Technology, 
424 Hafez Ave, Tehran, Iran} \email{bahareh.lajmiri@aut.ac.ir}

\thanks{This paper was written while the third author visited the Faculty of Mathematics and Computer Science, Ovidius University of Constanta. The research of the third author was partially supported by Amirkabir University of Technology.}

\begin{abstract}
We prove that $t$-spread principal Borel ideals are sequentially Cohen-Macaulay and study their powers. We show that these ideals 
possess the strong persistence property and compute their limit depth.
\end{abstract}


\maketitle
\section*{Introduction}

In this paper, we study $t$-spread prinipal Borel ideals. They have been recently introduced in \cite{EHQ}. Let $K$ be a field and $S=K[x_1,\ldots,x_n]$ the polynomial ring in $n$ variables over $K.$ Let $t\geq 0$ be an integer. A monomial $x_{i_1}\cdots x_{i_d}\in S$ with $i_1\leq \cdots \leq i_d$ is called \emph{$t$-spread} if $i_j-i_{j-1}\geq t$ for $2\leq j\leq d.$ For example, every monomial in $S$ is $0$-spread and every squarefree monomial is $1$-spread. 

We recall from \cite{EHQ} that a monomial ideal $I\subset S$ with the minimal system of monomial generators $G(I)$
is called \emph{$t$-spread strongly stable} if it satisfies the following condition: for all $u\in G(I)$ and $j\in \supp(u),$ if $i<j$ and $x_i(u/x_j)$ is $t$-spread, then $x_i(u/x_j)\in I.$  A monomial  ideal $I\subset S$ is called \emph{$t$-spread principal Borel} if there exists a monomial $u\in G(I)$ such that $I=B_t(u)$ where $B_t(u)$ denotes 
the smallest $t$-spread strongly stable ideal which contains $u.$ For example, for an  integer $d\geq 2,$ if 
$u=x_{n-(d-1)t}\cdots x_{n-t}x_n$, then $B_t(u)$ is minimally generated by all the 
$t$-spread monomials of degree $d$ in $S.$ In this case, we call $B_t(u)$ the \emph{$t$-spread Veronese ideal} generated in degree $d$ and denote it $I_{n,d,t}.$ 

Throughout this paper we consider $t$-spread monomial ideals with $t\geq 1. $ In particular, they are squarefree monomial ideals. As it was observed in \cite{EHQ}, if $u=x_{i_1}\cdots x_{i_d}$ is a $t$-spread monomial in $S$ then 
a monomial $x_{j_1}\cdots x_{j_d}\in G(B_t(u))$ if and only if the following two conditions hold:
\begin{enumerate}
	\item [(i)] $j_k\leq i_k$ for $1\leq k\leq d,$
	\item [(ii)] $j_k-j_{k-1}\geq t$ for $2\leq k\leq t.$
\end{enumerate}

In \cite[Theorem 1.4]{EHQ} it was shown that every $t$-spread strongly stable ideal has linear quotients with respect to the pure lexicographic 
order in $S.$ In addition, in \cite[Corollary 2.5]{EHQ} it was shown that a $t$-spread strongly stable ideal generated in a single degree  is Cohen-Macaulay if and only if it is a $t$-spread Veronese ideal. In particular, it follows that every $t$-spread principal Borel ideal has linear quotients and such an ideal is Cohen-Macaulay if and only if it is $t$-spread Veronese. 

Since $B_t(u)$  is a squarefree monomial ideal, we may interpret it as the Stanley-Reisner ideal of a simplicial complex $\Delta.$ In the 
proof of Theorem~\ref{ass}, we characterize the facets of $\Delta$. This  gives explicitly  all the generators of the ideal 
of the Alexander dual  $\Delta^{\vee}.$ We then derive that the Stanley-Reisner ideal of $\Delta^{\vee}$ has linear 
quotients, which yields the sequential Cohen-Macaulay property of  
$B_t(u)$ by Alexander duality; see Theorem~\ref{SCM}. 

In Section~\ref{two}, we consider the Rees algebra $\MR(B_t(u))$ of a $t$-spread principal Borel ideal $B_t(u).$ In 
Proposition~\ref{elexchange}, we show that  $B_t(u)$ satisfies the 
$\ell$-exchange property with respect to the sorting order $<_{sort}.$ This allows us to describe in Theorem~\ref{GBRees} the reduced 
Gr\"obner basis of the 
defining ideal of $\MR(B_t(u))$ with respect to a suitable monomial order. The form of the binomials in this Gr\"obner basis shows that 
$B_t(u)$ satisfies an $x$-condition which guarantees  that all the powers of $B_t(u)$ have linear quotients. Moreover, the Rees algebra 
$\MR(B_t(u))$ is a normal Cohen-Macaulay domain which implies in turn, by a result of \cite{HQ}, that $B_t(u)$ possesses  the strong 
persistence property as considered in \cite{HQ}. Consequently, $B_t(u)$ satisfies the persistence property, which means that 
$\Ass(B_t(u))\subset \Ass(B_t(u)^2)\subset \cdots \subset \Ass(B_t(u)^k)\subset \cdots.$ Note that the associated primes of $B_t(u)$ can be read from Theorem~\ref{ass}.

In Section~\ref{three}, we study the limit behavior of the depth for the powers of $t$-spread principal Borel ideals. In Theorem~\ref{limdepth} we show that if the generator $u=x_{i_1}\cdots x_{i_d}$ satisfies the condition $i_1\geq t+1,$ then 
$\lim_{k\to\infty}\depth(S/B_t(u)^k)=0.$ In particular, we determine the analytic spread of $B_t(u),$ that is, the Krull dimension of the fiber ring  
$\MR(B_t(u))/\mathfrak{m} \MR(B_t(u))$ where $\mathfrak{m}=(x_1,\ldots,x_n).$

\section{$t$-spread principal Borel ideals are sequentially Cohen-Macaulay}

Let $u=x_{i_1}x_{i_2}\cdots x_{i_d}\in S= K[x_1,\ldots, x_n]$ be a $t$-spread monomial. Suppose that $i_d=n$ and $I=B_t(u)$. For a monomial $u\in I$, we denote 
$\supp(u)=\{j\in [n]: x_j \mid u\}.$ Let $\Delta$ be the simplicial complex such that $I_{\Delta}=I$. We denote $I^{\vee}$ the Stanley-Reisner ideal of the Alexander dual of $\Delta$.

\begin{Theorem}\label{ass}
  Let $t\geq 1$ be an integer and $I=B_t(u)$ where $u=x_{i_1}x_{i_2}\cdots x_{i_d}$ is  a $t$-spread monomial. We assume that $\bigcup_{v\in G(I)}\supp(v)=[n].$ Then $I^{\vee}$ is generated by the monomials of the following forms
  \begin{equation}\label{form1}
 \prod_{k=1}^{n}x_k/(v_{j_1}\cdots v_{j_{d-1}})
\end{equation}
   with $j_l\leq i_l$ for $1\leq l\leq d-1$ and $j_l-j_{l-1}\geq t$ for $2\leq l\leq d-1$, where $v_{j_k}=x_{j_k}\cdots x_{j_k+(t-1)}\text{ for }1\leq k\leq d-1.$
\begin{equation}\label{form2}
\prod_{k=1}^{i_1}x_k.
\end{equation}
\begin{equation}\label{form3}
\prod_{k=1}^{i_s}x_k/(v_{j_1}\cdots v_{j_{s-1}})
\end{equation}
    with $2\leq s\leq d-1$, $j_l\leq i_l$ for $1\leq l\leq s-1$, $j_l-j_{l-1}\geq t$ for $2\leq l\leq s-1$, where 
		$v_{j_k}=x_{j_k}\cdots x_{j_k+(t-1)}\text{ for }1\leq k\leq s-1.$
\end{Theorem}

\begin{proof}
  Let $\Delta$ be the simplicial complex whose Stanley-Reisner ideal is $I$ and let $\mathcal{F}(\Delta)$ be the set of the facets of $\Delta$. We prove that every facet of $\Delta$ is  of one of the following forms:
  \begin{enumerate}\label{facets}
   \item[(i)] $F_1=\{j_1, j_1+1, \ldots, j_1+(t-1), j_2, j_2+1, \ldots, j_2+(t-1),\ldots, j_{d-1}, j_{d-1}+1, \ldots, j_{d-1}+(t-1)\}$, for some $j_1, j_2, \ldots, j_{d-1}$ such that $j_l\leq i_l$ for $1\leq l\leq d-1$ and $j_l-j_{l-1}\geq t$ for $2\leq l\leq d-1$.
   \item[(ii)] $F_2=\{i_1+1, i_1+2,\ldots, n\}$.
   \item[(iii)] $F_3=\{j_1, j_1+1, \ldots, j_1+(t-1),\ldots, j_{s-1}, j_{s-1}+1, \ldots, j_{s-1}+(t-1), j_s, j_s+1, \ldots, n\}$, for some $j_1, j_2, \ldots, j_s$ such that $2\leq s\leq d-1$, $j_l\leq i_l$ for $1\leq l\leq s-1$, $j_s=i_s+1$ and $j_l-j_{l-1}\geq t$ for $2\leq l\leq s$.
  \end{enumerate}
  
Since $I=B_t(u)$ has the primary decomposition \[I=\bigcap_{F\in\mathcal{F}(\Delta)}P_{[n]\setminus F},\]
  where $P_{[n]\setminus F}$ is the  prime ideal generated by all variables $x_j$ with $j\in[n]\setminus F$, by 
	\cite[Corollary 1.5.5]{HHBook}, the statement holds.

  Since $(x_1,\ldots x_{i_1})\in \Min(I)$ by \cite[Theorem 2.4]{EHQ}, we obtain $F_2\in\mathcal{F}(\Delta)$.
  
  We have $F_1$ and $F_3\in \Delta$, since $x_{F_1}=\prod_{i\in F_1}x_i$ and $x_{F_3}=\prod_{i\in F_3}x_i$ do not belong to $I$. Indeed, if $x_{F_1}\in I$, then there exists $v=x_{k_1}\cdots x_{k_d}\in G(I)$ such that $v\mid x_{F_1}$. Since $v$ is the product of $d$ distinct variables and $F_1$ consists of $d-1$ intervals of the form $[j_r, j_r+(t-1)],$ $1\leq r\leq d-1$, there exist $l\in \{2,\ldots, d\}$ and $1\leq r\leq d-1$ such that $k_{l-1}, k_l\in [j_r,j_r+(t-1)]$. It follows that $k_l-k_{l-1}<t$, which is false.

  If $x_{F_3}\in I$, then there exists $v=x_{k_1}\cdots x_{k_d}\in G(I)$ such that $v\mid x_{F_3}$. Since $k_l-k_{l-1}\geq t$ for $2\leq l\leq s$ and $F_3$ consists of $s-1$ intervals of the form $[j_r, j_r+(t-1)],$ $1\leq r\leq s-1,$ and only one interval of the form $[j_s,n]$, we have $k_s\in [j_s,n]$.  It follows that $j_s=i_s+1\leq k_s$, which is false because $v\in G(I)$ and $k_s\leq i_s$.

  We claim that $F_i\cup \{j\}\notin \Delta$ for every $j\in [n]\setminus F_i$ and $i\in\{1,3\}$. This will prove that every set of the form (i) or (iii) is a facet of $\Delta.$

  Let $j\in[n]\setminus F_1$. We consider the following cases:
  \begin{enumerate}
    \item[(a)]  $j<j_1.$ Then $x_j x_{j_1+(t-1)}\cdots x_{j_{d-1}+(t-1)}\in I$ because $j<j_1\leq i_1$ and $j_l+(t-1)\leq i_l+(t-1)<i_{l+1}$ for $1\leq l\leq d-1$.
    \item[(b)]  $j\geq j_{d-1}+t.$ Then $x_{j_1} x_{j_2}\cdots x_{j_{d-1}}x_j\in I$.
    \item[(c)] There exists $1\leq l\leq d-2$ such that $j_l+(t-1)<j<j_{l+1}.$ Then \[x_{j_1} x_{j_2}\cdots x_{j_l} x_j x_{j_{l+1}+(t-1)}\cdots  x_{j_{d-1}+(t-1)}\in I,\] since $j_k\leq i_k$ for $1\leq k\leq l$, $j<j_{l+1}\leq i_{l+1}$ and $j_k+(t-1)\leq i_k+(t-1)<i_{k+1}$ for $l+1\leq k\leq d-1$.
  \end{enumerate}
	Therefore, we have proved that for $j\in[n]\setminus F_1$, $F_1\cup\{j\}\not\in \Delta$ which implies that $F_1$ is a facet in $\Delta.$

  Let $j\in[n]\setminus F_3$. We show that $x_j x_{F_3}\in I$, thus $F_3\cup\{j\}\not\in \Delta.$
  \begin{enumerate}
    \item[(a)] If $j<j_1$, then $x_j x_{j_1+(t-1)}\cdots x_{j_{s-1}+(t-1)}x_{j_s+(t-1)}x_{j_s+(2t-1)}\cdots x_{j_s+(d-s)t-1}\in I$ because $j<j_1\leq i_1$, $j_k+(t-1)\leq i_k+(t-1)<i_{k+1}$ for $2\leq k\leq s-1$ and $j_s+(kt-1)=i_s+1+(kt-1)=i_s+kt\leq i_{s+k}$ for $1\leq k\leq d-s$.
    \item[(b)] If there exists $2\leq l\leq s$ such that $j_{l-1}+(t-1)<j<j_l$, then \[x_{j_1} x_{j_2}\cdots x_{j_{l-1}} x_j x_{j_l+(t-1)}\cdots  x_{j_s+(t-1)}\cdots x_{j_s+(d-s)t-1}\in I.\]
        Indeed, $j_k\leq i_k$ for $1\leq k\leq l$. If $l=s$, then $j<j_s=i_s+1$ and $j_s+(kt-1)=i_s+1+(kt-1)=i_s+kt\leq i_{s+k}$ for 
				$1\leq k\leq d-s$. In the case that $l<s$, we have $j<j_l\leq i_l$, $j_k+(t-1)\leq i_k+(t-1)<i_{k+1}$ for $l\leq k\leq s-1$ and $j_s+(kt-1)=i_s+1+(kt-1)\leq i_{s+k}$ for $1\leq k\leq d-s$.
  \end{enumerate}

  It remains to show that the sets of the forms (i)--(iii) are the only facets of $\Delta$. This is equivalent to showing that for every face 
	$G\in \Delta$, there exists $F\in \mathcal{F}(\Delta)$ of one of the forms (i)--(iii)  which contains $G$.

  Let $G\in \Delta$ and $j_1=\min\{j:j\in G\}$. Inductively, for $l\geq 2$, we set \[j_l=\min\{j:j\in G\text{ and }j\geq j_{l-1}+t\}.\]

  In the case that $j_l\leq i_l$ for all $l$, then the sequence $j_1<j_2<\ldots<j_l<\ldots$ has at most $d-1$ elements. Otherwise, $x_{j_1}\cdots x_{j_d}\in I$, which implies that $G\notin\Delta$, a contradiction. Let $j_1<j_2<\ldots<j_r$ with $r\leq d-1$. Then $G\subset F_1=\{j_1,\ldots, j_1+(t-1),\ldots, j_r,\ldots, j_r+(t-1), i_{r+1},\ldots, i_{r+1}+(t-1),\ldots, i_{d-1},\ldots, i_{d-1}+(t-1)\}\in \mathcal{F}(\Delta)$.

  If there exists $l\leq d$ such that $j_l>i_l$, then we denote by $s$ the smallest index with this property. In the case that $s=1$, $j_1=\min\{j:j\in G\}$ and $G\subset F_2=\{i_1+1, i_1+2,\ldots, n\}\in \mathcal{F}(\Delta)$. If $s>1$, then $G\subset F_3=\{j_1, j_1+1, \ldots, j_1+(t-1),\ldots, j_{s-1}, j_{s-1}+1, \ldots, j_{s-1}+(t-1), i_s+1, i_s+2, \ldots, n\}\in \mathcal{F}(\Delta)$.
\end{proof}

\begin{Theorem}\label{SCM}
Let $t\geq 1$ be an integer, $u=x_{i_1}x_{i_2}\cdots x_{i_d}\in S$  a $t$-spread monomial, and  suppose that $i_d=n.$ Then
the $t$-spread principal Borel ideal  $I=B_t(u)$ is sequentially Cohen-Macaulay.
\end{Theorem}

\begin{proof}
By \cite[Theorem 8.2.20]{HHBook}, it is enough to show that $I^{\vee}$, that is, the Stanley-Reisner ideal of $\Delta^{\vee}$, is componentwise linear. By \cite[Corollary 8.2.21]{HHBook}, it suffices to prove that $I^{\vee}$ has linear quotients.  

Let $w_1=\prod_{k=1}^{i_1}x_k$. Let $w_2,w_3,\ldots, w_q$ be the minimal monomial generators of $I^{\vee}$  of the form (\ref{form3}), ordered decreasingly with respect to the  pure lexicographic order. Let $w_{q+1},w_{q+2},\ldots, w_r$ be the remaining minimal monomial generators of $I^{\vee}$, namely those of form (\ref{form1}), ordered decreasingly with respect to the pure lexicographic order.
  
We show that $I^{\vee}$ has linear quotients with respect to the above order of its minimal generators. This is equivalent to proving the following claim:
\begin{eqnarray}\label{eqn}
 \text{ for every }1<j\leq r \text{ and }g<j, \text{ there exists }w>_{\lex}w_j \text{ and  } \\ \nonumber
1\leq l\leq n \text{  such that }x_l=\frac{w}{\gcd(w, w_j)} \text{ and }x_l\mid \frac{w_g}{\gcd(w_g, w_j)}.
\end{eqnarray}

Since $w_2$ is the largest generator of the form (\ref{form3}) with respect to the pure lexicographic order, we have $x_1x_2\cdots x_{i_1-1}\mid w_2$. Thus, we obtain $x_{i_1}=w_1/\gcd(w_1, w_2)$ and claim (\ref{eqn}) is proved for $g=1<j=2.$
  
Let $2<j\leq q$, $g=1$ and $w_j=\prod_{k=1}^{i_s}x_k/v_{j_1}v_{j_2}\cdots v_{j_{s-1}}$ with $s\leq d-1$. 
If $j_1=i_1$, then we have  $x_{i_1}={w_1}/{\gcd(w_1, w_j)}.$  Let $j_1<i_1$. Then $x_{j_1}\mid w_1/\gcd(w_1, w_j)$. We have to find $w>_{\lex}w_j$ such that $x_{j_1}=w/\gcd(w, w_j)$.
  
	\begin{enumerate}
    \item[(a)] If  there exists a least integer $1\leq l\leq s-2$ such that $j_{l+1}>j_l+t$, then we take $w=\prod_{k=1}^{i_s}x_k/v_{j_1+1}\cdots v_{j_l+1}v_{j_{l+1}}\cdots v_{j_{s-1}}$. 
		\item[(b)] If  $j_{l+1}=j_l+t$ for $1\leq l\leq s-2$, then \[w=\prod_{k=1}^{i_s}x_k/v_{j_1+1}\cdots v_{j_l+1}v_{j_{l+1}+1}\cdots v_{j_{s-1}+1}.\]
  \end{enumerate}
  
  Let $1<g<j$. Then $w_g>_{\lex}w_j$. Let \[w_g=\frac{\prod_{k=1}^{i_{s'}}x_k}{v_{g_1}v_{g_2}\cdots v_{g_{s'-1}}}\] with $s'\leq d-1$. Thus, we obtain $v_{g_1}v_{g_2}\cdots v_{g_{s'-1}}<_{\lex}v_{j_1}v_{j_2}\cdots v_{j_{s-1}}$. It implies that there exists $h\geq 1$ such that $j_1=g_1$, $j_2=g_2$,\ldots, $j_{h-1}=g_{h-1}$ and $j_h<g_h\leq i_h$. Since $x_{j_h}\mid v_{j_1}v_{j_2}\cdots v_{j_{s-1}}$ and $x_{j_h}\nmid v_{g_1}v_{g_2}\cdots v_{g_{s'-1}}$, we get $x_{j_h}\mid {w_g}/{\gcd(w_g,w_j)}.$ Then we must find a generator $w$ of 
	$I^{\vee}$ with $w>_{\lex}w_j$ such that $x_{j_h}=w/\gcd(w,w_j).$
\begin{enumerate}	
\item[(a)] If there exists a least integer $h\leq l\leq s-2$ such $j_{l+1}>j_l+t$, then
  \[w=\frac{\prod_{k=1}^{i_s}x_k}{v_{j_1}\cdots v_{j_{h-1}}v_{j_h+1}v_{j_{h+1}+1}\cdots v_{j_l+1}v_{j_{l+1}}\cdots v_{j_{s-1}}}.\]
\item[(b)] If  $j_{l+1}=j_l+t$ for $h\leq l\leq s-2$, then $w={\prod_{k=1}^{i_s}x_k}/{v_{j_1}\cdots v_{j_{h-1}}v_{j_h+1}\cdots v_{j_{s-1}+1}}.$
\end{enumerate} 
  With  similar arguments, one proves claim (\ref{eqn}) for $q+1\leq j\leq r$.
\end{proof}

\begin{Example}
Let $I=B_2(x_2x_4x_9)\subset K[x_1,\ldots,x_9].$ Then  \[I^{\vee}=(x_1x_2, x_1x_4, x_3x_4, x_1x_6x_7x_8x_9, x_3x_6x_7x_8x_9, x_5x_6x_7x_8x_9),\] where we ordered the generators as indicated in the proof of the above theorem.
\end{Example}

\section{The Rees algebra of $t$-spread Borel principal ideals}
\label{two}

In this section we consider the 
Rees algebra of  $t$-spread Borel principal ideals. 

For two monomials $v,w\in S$ of degree $d$, we write $vw=x_{i_1}x_{i_2}\cdots x_{i_{2d}}$ with $i_1\leq i_2 \leq \cdots \leq i_{2d}.$
Then the \emph{sorting} of the pair $(v,w)$ is the pair of monomials $(v',w')$ where $v'=x_{i_1}x_{i_3}\cdots x_{i_{2d-1}}$ and $w'=x_{i_2}x_{i_4}\cdots x_{i_{2d}}.$
The map $\sort:S_d \times S_d \to S_d \times S_d$
with $\sort(v,w)=(v',w')$ is called the \emph{sorting operator}. A subset $B \subset S_{d}$ is called  \emph{sortable} if $\sort(B \times B) \subset B\times B$. An $r$-tuple $(u_1,\ldots,u_r)$
 of monomials  of degree $d$ is sorted if for all $1\leq i<j\leq r,$ the pair $(u_i,u_j)$ is sorted, that is, 
$\sort(u_i,u_j)=(u_i,u_j).$ This means that if we write $u_1=x_{i_1}\cdots x_{i_d}, u_2=x_{j_1}\cdots x_{j_d},\ldots, u_r=x_{k_1}\cdots x_{k_d},$ then $(u_1,\ldots,u_r)$ is sorted if and only if $i_1\leq j_1\leq \cdots \leq k_1\leq i_2\leq j_2\leq \cdots \leq k_2\leq \cdots \leq i_d\leq j_d\leq \cdots\leq k_d.$ By \cite[Theorem 6.12]{EHBook}, for every $r$-tuple $(u_1,\ldots,u_r)$
 of monomials of degree $d$, there exists a unique sorted $r$-tuple $(u_1^\prime,\ldots,u_r^\prime)$ such that 
$u_1\cdots u_r=u_1^\prime\cdots u_r^\prime.$ When $(u_1^\prime,\ldots,u_r^\prime)=(u_1,\ldots,u_r),$ we say that the product $u_1\cdots u_r$ is sorted.

Let $t\geq 1$ be an integer and $I=B_t(u)$ where $u=x_{i_1}\cdots x_{i_d}$ is a $t$-spread monomial.
It was shown in \cite[Proposition 3.1]{EHQ} that the minimal set of monomial generators of a $t$-spread principal Borel ideal is sortable.
Therefore, if $v,w \in G(I)$, and $\sort(v,w)=(v',w')$, then $v',w' \in G(I)$.

Let $\MR(I)= \bigoplus_{j\geq 0} {I^j}{t^j}$ be the Rees algebra of the ideal $I=B_t(u).$ 
Since the minimal generators of $B_t(u)$ have the same degree, the fiber 
$\MR(I)/\mathfrak{m}\MR(I)$ of the Rees ring $\MR(I)$ is isomorphic to $K[G(I)].$ Let $\psi:T=K[\{t_v: v\in G(I)\}]\to K[G(I)]$ be the homomorphism  given by 
$t_{v}\mapsto v$ for $v\in G(I).$ 
 As it was proved in  \cite[Theorem 3.2]{EHQ}, the set of binomials 
\[\MG=\{t_vt_w-t_{v^\prime}t_{w^\prime}: (v,w) \text{ unsorted }, (v^\prime, w^\prime)=\sort(v,w)\}\]
 is a
Gr\"obner basis of the toric ideal $J_u=\ker \psi$ with respect to the sorting order on $T$. For more details about  sorting order we refer to 
\cite[Chapter 14]{St95} or \cite[Section 6.2]{EHBook}. The initial monomial of the binomial $t_vt_w-t_{v^\prime}t_{w^\prime}\in \MG$ with respect to the sorting order is $t_vt_w.$

We now  recall the definition of \emph{$\ell$-exchange property} from \cite{HHV05}; see also \cite[Section 6.4]{EHBook} and
\cite[Section 3]{EO}. 

Let $I\subset S$ be a monomial ideal generated in a single degree and $K[\{t_u: u\in G(I)\}]$ the polynomial ring in $|G(I)|$ 
variables endowed with a monomial order $<.$ Let $P$ be the kernel of the $K$-algebra homomorphism 
 \[ K[\{t_u: u\in G(I)\}]\to K[G(I)], t_u\mapsto u, u\in G(I).\]

A monomial $t_{u_1}\cdots t_{u_N}$ is called \emph{standard with respect to $<$} if it does not belong to $\ini_<(P).$

\begin{Definition}\cite{HHV05}\label{defellexchange}
The monomial ideal $I\subset S$ satisfies the $\ell$-exchange property with respect to $<$ if the following condition holds: let 
 $t_{u_1}\cdots t_{u_N},t_{v_1}\cdots t_{v_N}$ be two standard monomials  with respect to $<$  of the same degree $N$
satisfying 
\begin{itemize}
	\item [(i)] $\deg_{x_i} u_1\cdots u_N=\deg_{x_i}v_1\cdots v_N$ for $1\leq i\leq q-1$ with $q\leq n-1,$
	\item [(ii)] $\deg_{x_q}u_1\cdots u_N<\deg_{x_q}v_1\cdots v_N.$
\end{itemize}
Then there exists  integers $\delta, j$ with $q<j\leq n$ and $j\in \supp(u_\delta)$ such that $x_q u_\delta/x_j\in I.$
\end{Definition}

\begin{Proposition}\label{elexchange}
Let $u=x_{i_1}\cdots x_{i_d}$ be a $t$-spread monomial in $S.$ Then the $t$-spread principal Borel ideal $B_t(u)$ satisfies the 
$\ell$-exchange property with respect to the sorting order $<_{sort}.$
\end{Proposition}

\begin{proof}
Let $T=K[\{t_v : v\in G(B_t(u))\}]$ and $t_{u_1}\cdots t_{u_N},t_{v_1}\cdots t_{v_N}\in T$ be two standard monomials  with respect to $<_{\sort}$  of the same degree $N$
such that the two conditions of Definition~\ref{defellexchange} are fulfilled. As the chosen monomials are standard with respect to $<_{\sort}$, it follows that the products 
$ u_1\cdots u_N, v_1\cdots v_N$ are sorted. Since $\deg_{x_i} u_1\cdots u_N=\deg_{x_i}v_1\cdots v_N$ for $1\leq i\leq q-1$, we have $\deg_{x_i}(u_\nu)=\deg_{x_i}(v_\nu)$
for all $1\leq \nu\leq N$ and $1\leq i\leq q-1$. The condition $\deg_{x_q}u_1\cdots u_N<\deg_{x_q}v_1\cdots v_N$ implies that there exists 
$1\leq \delta\leq N$ such that $\deg_{x_q}(u_\delta)<\deg_{x_q}(v_\delta).$

Let $u_{\delta}=x_{j_1}\cdots x_{j_d}$, $v_{\delta}=x_{\ell_1}\cdots x_{\ell_d}$, and assume that $q=\ell_\mu$ for some $1\leq \mu <d.$ It follows that 
$j_1=\ell_1,\ldots,j_{\mu-1}=\ell_{\mu-1}$ and $j_\mu>\ell_\mu=q.$ If $j_\mu\not\in \supp(v_{\delta}),$ then we take $j=j_\mu.$ Then the monomial 
$x_qu_{\delta}/x_j$ is $t$-spread, thus it belongs to $B_t(u)$, since $q=\ell_\mu\geq \ell_{\mu-1}+t=j_{\mu-1}+t$ and $j_{\mu+1}\geq j_\mu+t>q+t.$

If $j_\mu\in \supp(v_{\delta}), $ we may choose any $j\in \supp(u_\delta)\setminus \supp(v_{\delta})$ with $j>q.$ We should note that such an integer $j$ exists since
$\deg (u_\delta)=\deg (v_\delta)$ and $\deg_{x_q}(u_\delta)<\deg_{x_q}(v_\delta).$ Then $x_qu_{\delta}/x_j$ is $t$-spread since $q,j_\mu\in\supp(v_{\delta}).$
\end{proof}

The Rees ring $\MR(I)$ where $I=B_t(u)$ has the presentation
\[
\varphi: R=S[\{t_v : v\in G(I)\}]\to \MR(I),
\]
given by
\[
x_i\mapsto x_i, 1\leq i\leq n, t_v\mapsto v t, v\in G(I).
\]
Let $J=\ker \varphi$ be the toric ideal of $\MR(I).$ We consider the sorting order $<_{\sort}$ on the ring $T=K[\{t_v : v\in G(I)\}]$ and the lexicographic order $<_{\lex}$ on the ring $S.$
Let $<$ be the monomial order on $ R=S[\{t_v : v\in G(I)\}]$ defined as follows: if $m_1, m_2$ are monomials in $S$ and $v_1,v_2$ are monomials in $T,$ then 
\[
m_1v_1>m_2v_2 \text{ if } m_1>_{\lex} m_2 \text{ or } m_1=m_2 \text{ and } v_1>_{\sort}v_2.
\]

\begin{Theorem}\label{GBRees} Let $I=B_t(u)$ be a $t$-spread principal Borel ideal.
The reduced Gr\"obner basis of the toric ideal $J$ with respect to $<$ consists of the set of binomials $t_vt_w-t_{v^\prime}t_{w^\prime}$ where $(v,w)$ is  unsorted and 
 $(v^\prime, w^\prime)=\sort(v,w),$ together with
the binomials of the form $x_i t_v-x_j t_w$ where $i<j,$ $x_iv=x_jw$, and $j$ is the largest integer for which $x_iv/x_j\in G(I).$
\end{Theorem}

\begin{proof}
By Proposition~\ref{elexchange}, we know that the principal Borel ideal $B_t(u)$ satisfies the 
$\ell$-exchange property with respect to the sorting order $<_{sort}.$ Thus, the statement follows by applying \cite[Theorem 5.1]{HHV05}.
\end{proof}

\begin{Proposition}\label{linear quotients}
All the powers of $B_t(u)$ have linear quotients. In particular, all the powers of $B_t(u)$ have a linear resolution.
\end{Proposition}

\begin{proof}
By Theorem~\ref{GBRees}, it follows that every binomial in the reduced Gr\"obner basis of the toric ideal $J$  has the degree in variables $x_i$ at most $1$ which means that $B_t(u)$
satisfies the $x$-condition. Then, by \cite[Theorem 10.1.9]{HHBook}, all the powers of $B_t(u)$ have linear quotients. 
\end{proof}

Theorem~\ref{GBRees} shows that the Rees algebra $\MR(B_t(u))$ has a quadratic Gr\"obner basis and the initial ideal of the toric ideal $J$ is squarefree. Therefore, we get the following 
consequences.

\begin{Corollary}
The Rees algebra $\MR(B_t(u))$ is Koszul.
\end{Corollary}

\begin{Corollary}\label{CMRees}
The Rees algebra $\MR(B_t(u))$ is a normal Cohen-Macaulay domain. In particular, $B_t(u)$ satisfies the strong persistence property.
Therefore, $B_t(u)$ satisfies the persistence property.
\end{Corollary}

\begin{proof}
Since the initial ideal of the toric ideal $J$ of $\MR(B_t(u))$ is squarefree, by a theorem due to Sturmfels \cite{St95}, it follows that $K[B_t(u)]$ is a normal domain. Next, by a theorem of Hochster \cite{Ho72}, it follows that
$K[B_t(u)]$ is Cohen-Macaulay.  The second part follows by using \cite[Corollary 1.6]{HQ}.
\end{proof}

\section{Depth of powers of $t$-spread principal Borel ideals}
\label{three}

In this section, we consider the asymptotic behavior of the depth for $t$-spread principal Borel ideals. 

Let $u=x_{i_1}\cdots x_{i_d}$ be the generator of the $t$-spread principal Borel ideal $I=B_t(u)\subset S=K[x_1,\ldots,x_n].$ We would like to compute the limit of 
$\depth(S/I^k)$ when $k\to \infty.$
Obviously, we may consider $i_d=n$ since otherwise we may reduce to the study of $\depth S'/I^k$ where $S' =K[x_1,\ldots,x_{i_d}].$ 

\begin{Theorem}\label{limdepth}
Let $t\geq 1$ be an integer and $I=B_t(u)\subset S$ the $t$-spread principal Borel ideal generated by $u=x_{i_1}\cdots x_{i_d}$ where $t+1\leq i_1<i_2<\cdots <i_{d-1}<i_d=n.$ Then 
\[
\depth\frac{S}{I^k}=0, \text{ for }k\geq d.
\]
In particular, the analytic spread of $I$ is $\ell(I)=n.$
\end{Theorem}

\begin{proof}
We first observe that the second claim of the theorem follows since it is known that, for any non-zero graded ideal $I,$ $\lim_{k\to \infty} \depth(S/I^k)=n-\ell(I)$ if the Rees algebra $\MR(I)$
is Cohen-Macaulay; see \cite[Proposition 10.3.2]{HHBook}. But the Rees algebra of $B_t(u)$ is indeed Cohen-Macaulay by Corollary~\ref{CMRees}.

In order to prove the first claim, by Auslander-Buchsbaum theorem, we have to show that 
\begin{equation}\label{eq1}
\projdim \frac{S}{I^k}=n \text{ for }k\geq d.
\end{equation}

Since $I$ satisfies the $\ell$-exchange property (Proposition~\ref{elexchange}), by using \cite[Theorem 3.6]{EO}, it follows that $I^k$ has linear quotients with respect to $>_{\lex}$
for all $k\geq 1.$ This implies that we may compute the projective dimension of $S/I^k$ by using \cite[Corollary 8.2.2]{HHBook}. More precisely, if $G(I^k)=\{w_1,\ldots,w_m\}$ 
with $w_1>_{\lex}\cdots >_{\lex}w_m$ and for $1\leq j\leq m,$ $r_j$ is the number of variables which generate the ideal quotient $(w_1,\ldots,w_{j-1}):w_j,$ then
\[
\projdim \frac{S}{I^k}=\max\{r_1,\ldots,r_m\}+1.
\]

Thus, in order to prove (\ref{eq1}), we only need to find a monomial $w\in G(I^k)$ such that the ideal quotient $(w'\in G(I^k): w'>_{\lex} w):w$ is generated by  $n-1$ variables.
Let $w_0=v_1 v_2\cdots v_{d-1}v_d$ where 
\[
v_1=x_1 x_{t+1} \ldots x_{(d-3)t+1} x_{(d-2)t+1} x_n\]
\[
v_2=x_1 x_{t+1} \ldots x_{(d-3)t+1} x_{i_{d-1}} x_n
\]
\[\vdots \]
\[
v_{d-1}=x_1 x_{i_2} \ldots x_{i_{d-2}} x_{i_{d-1}} x_n\]
\[
v_d=u=x_{i_1} x_{i_2} \ldots x_{i_{d-2}} x_{i_{d-1}} x_n.\]

Obviously, all the monomials $v_1,\ldots,v_d$ belong to $G(I).$ Let $k\geq d$ and $w=w_0u^{k-d}\in G(I^k).$ If we show that 
$(w'\in G(I^k): w'>_{\lex} w):w\supseteq (x_1,\ldots,x_{n-1})$ then the proof 
of equality (\ref{eq1}) is completed. 

Let $J=(w'\in G(I^k): w'>_{\lex} w).$ We show that $x_j w\in J$ for all $1\leq j\leq n-1.$ Let $1\leq s\leq d$ such that $i_{s-1}\leq j <i_s,$ where we set $i_0=1.$
We consider the monomial 
\[
v'_{d-s+1}=\frac{x_j v_{d-s+1}}{x_{i_s}}=x_1x_{t+1}\cdots x_{(s-2)t+1}x_j x_{i_{s+1}}\cdots x_{i_{d-1}}x_n.
\]
Clearly, $v'_{d-s+1}\in G(I)$ since $j<i_s$ and $v'_{d-s+1}$ is a $t$-spread monomial. Indeed, we have $i_{s+1}-j>i_{s+1}-i_s\geq t$ and 
\[j-(s-2)t-1\geq i_{s-1}-(s-2)t-1\geq (i_1+(s-2)t)-(s-2)t-1=i_1-1\geq t.
\]
Let $w'_0$ be monomial obtained from  $w_0$ by replacing the monomial $v_{d-s+1}$ with $v'_{d-s+1}$ and let $w'=w'_0 u^{k-d}.$ Then $w'>_{\lex}w,$ thus $w'\in J$ and since $x_j w=x_{i_s}w',$ 
we have $x_jw\in J,$ which implies that $x_j\in J:w.$
\end{proof}

The analytic spread $\ell(I)$ of a graded ideal $I\subset S$ is the Krull dimension of the fiber ring $\MR(I)/\mathfrak{m} \MR(I)$ where $\mathfrak{m}=(x_1,\ldots,x_n).$ In our case, $I=B_t(u)$ is generated in the same degree $d.$ Thus the fiber ring is actually $K[G(B_t(u))].$  Therefore, we get the following consequence. 

\begin{Corollary}
Let $t\geq 1$ be an integer and $B_t(u)\subset S$ the $t$-spread principal Borel ideal generated by $u=x_{i_1}\cdots x_{i_d}$ where $t+1\leq i_1<i_2<\cdots <i_{d-1}<i_d=n.$
Then $\dim K[G(B_t(u))]=n.$
\end{Corollary}

\begin{Remark}
Note that the order with respect to which we derived that all the powers of $B_t(u)$ have linear quotients in Proposition~\ref{linear quotients} does not coincide with the decreasing lexicographic order that we used in the proof of Theorem~\ref{limdepth}. In fact, the toric ideal of 
$K[B_t(u)]$ does not have a quadratic Gr\"obner basis with respect to the lexicographic order as we can see in the following example. 

Let $u=x_6 x_8 x_{10}\in K[x_1,\ldots,x_{10}]$ and $B_2(u)$ the $2$-spread principal Borel ideal defined by $u$. Let 
$f=t_{u_1}t_{u_2}t_{u_3}-t_{v_1}t_{v_2}t_{v_3}$ with $u_1=x_1x_3x_8$, $u_2=x_1x_7x_9$, $u_3=x_2x_4x_6$ and $v_1=x_1x_3x_9$, $v_2=x_1x_6x_8$, $v_3=x_2x_4x_7$. Then $f$ is a binomial in the toric ideal of $K[B_2(u)]$ whose initial monomial with respect to the lexicographic order is $t_{u_1}t_{u_2}t_{u_3}.$ One can easily see that there is no quadratic monomial in the initial ideal of the toric ideal which divides 
$t_{u_1}t_{u_2}t_{u_3}.$ This shows that, with respect to the lexicographic order, the reduced Gr\"obner basis of the toric ideal of 
$K[B_2(u)]$ is not quadratic. 
\end{Remark}

\begin{Remark}
Condition $i_1\geq t+1$ in the hypothesis of Theorem~\ref{limdepth} is essential. Indeed, let us consider 
$I=B_t(u)$ where $u=x_t x_{2t}\cdots x_{dt}$, the $t$-spread Veronese ideal in $n=dt$ variables. We claim that 
\[
\depth \frac{S}{I^k}=d-1, \text{ for }k\geq d,
 \] which implies that 
\[
\lim_{k\to\infty} \depth \frac{S}{I^k}=d-1.
\]
In particular, this implies that the analytic spread of $I$ is  $\ell(I)=n-d+1=\dim K[B_t(u)].$
Indeed, we show that 
\[
\projdim \frac{S}{I^k}=n-d+1, \text{ for }k\geq d.
\] As in the proof of Theorem~\ref{limdepth}, one  considers the monomials \[v_1=x_1x_{t+1}\cdots x_{(d-2)t+1}x_{dt},\ldots,
v_{d-1}=x_1 x_{2t}\cdots x_{dt}, v_d=u=x_t x_{2t}\cdots x_{dt},\] and shows that, for $k\geq d,$ 
\[(w^\prime\in G(I^k):w^\prime >_{\lex} w):w\supseteq (x_j: j\in [n]\setminus \supp(u)),
\] where $w=v_1\cdots v_d u^{k-d}.$

On the other hand, we show that for every $j\in \supp(u),$ $x_j$ does not appear among the variables which generate the colon ideals 
 of $I^k.$ Let us assume that there exists $k\geq 1$ and some monomial $\mu\in I^k,$ $\mu=\mu_1\cdots \mu_k$ with $\mu_1,\ldots, \mu_k\in I$, 
such that $x_j\in (\mu^\prime\in I^k: \mu^\prime>_{\lex }\mu)$ for some $j\in \supp(u).$ This implies that there exists 
$\mu^\prime=\mu_1^\prime\cdots \mu_k^\prime\in I^k$ 
and some integer $s>j$ such that $x_j \mu= \mu^\prime x_s.$  Let $j=qt$ for some $1\leq q\leq d. $
Then
\[
\sum_{\ell=(q-1)t+1}^{qt}\deg_{x_{\ell}}(x_j \mu)=\sum_{\ell=(q-1)t+1}^{qt}\deg_{x_{\ell}}(x_j \mu_1\cdots \mu_k)=k+1> k=
\]
\[
=\sum_{\ell=(q-1)t+1}^{qt}\deg_{x_{\ell}}( \mu_1^\prime\cdots \mu_k^\prime x_s)=\sum_{\ell=(q-1)t+1}^{qt}\deg_{x_{\ell}} (\mu^\prime x_s),
\] contradiction. Therefore, the maximal number of variables which generate the colon ideals of $I^k$ for $k\geq d$ is $n-d,$ which implies that $\projdim \frac{S}{I^k}=n-d+1, \text{ for }k\geq d.$

\end{Remark}


\end{document}